\documentclass[10pt,a4paper]{amsart}

\usepackage[a4paper]{geometry}
\usepackage{ifxetex}
\ifxetex
  \relax
\else
   \usepackage[utf8]{inputenc} % Use 8-bit encoding that has 256 glyphs
\fi
\usepackage{amsthm,amsfonts,amssymb,amsmath,amscd,amsbsy}
\usepackage{mathtools} % AMSTEX fix and extension
\mathtoolsset{mathic=true} % Italic correction to math with \( \)
\usepackage{mathrsfs} %   Provides a \mathscr command Ralph Smith's
 \ifxetex
  \usepackage{fourier} % Pour les maths seulement
  \usepackage[no-math]{fontspec}
 \else
    \usepackage{erewhon}
 \fi

\ifxetex
  \usepackage[protrusion=true]{microtype} % Slightly tweak font spacing for aesthetics
  \else
  \relax
\fi

\usepackage[usenames,dvipsnames,svgnames,table]{xcolor}
\usepackage{hyperref} \hypersetup{%
  citecolor=NavyBlue, % color of links to bibliography
  colorlinks=true, % false: boxed links; true: colored links
  filecolor=magenta, % color of file links
  linkcolor=Maroon, % color of internal links
  pdfauthor={Livio Flaminio and Giovanni Forni}, % author
  pdfkeywords={horocycle flow, cohomology of flows}, % list of keywords
  pdfsubject={Dynamical systems}, % subject of the document
  pdftitle={Orthogonal powers  and Moebius conjecture},% Title
  urlcolor=cyan, % color of external links,
  anchorcolor=cyan % color for anchor text
}
\usepackage{graphicx}
\newtheorem{theorem}{Theorem}
\newtheorem{lemma}{Lemma}

\title[Orthogonal powers  and Möbius 
conjectur]{Orthogonal powers  and Möbius 
conjecture \\ for smooth time changes of horocycle flows}
\author{Livio Flaminio and Giovanni Forni}
\def\D{\mathrm d} 
\def\sl{\operatorname{\mathfrak s\mathfrak
  l_2}}
\def\R{\mathbb R}
\def\norm#1{\lVert#1\rVert}
\begin{document}
\begin{abstract}
  \begin{sloppypar}
   We derive, from the work of M.~Ratner on joinings of time-changes of horocycle flows and  from the result of the authors on its cohomology, the property of orthogonality of powers  for non-trivial smooth time-changes of horocycle flows on compact quotients.
 Such a property is known to imply P.~Sarnak's  Möbius orthogonality conjecture, already known for horocycle
 flows by the work of J.~Bourgain, P.~Sarnak and T.~Ziegler. 
   
  \end{sloppypar}
 \end{abstract}
\maketitle

\section{Introduction}

We set, for brevity, $G=\operatorname{PSL}_2(\R)$ and denote by
$\Gamma$ a co-compact lattice of $G$. We denote $h_t$ the classical
horocycle flow on $\Gamma\backslash G$. Let
$\tau \in W^s(\Gamma\backslash G)$ be a  strictly positive 
function of Sobolev order $s>2$, and let $h^\tau_t$ be the corresponding
time change of $h_t$. We recall that the flow $h^\tau_t$  is defined
by setting, for any $x\in \Gamma\backslash G$ and
$t\in\R$, 
\[
  h^\tau_t(x) :=h_{w(x,t)}(x), 
\]
with $w(x,t)$ the unique function satisfying the identity
\[
\int_0^{w(x,t)}\tau(h_ux) \, \D u =t
\]
for all $(x,t)\in \Gamma\backslash G\times \R$.
\begin{theorem}\label{thm:ratner_cohomology:1}
  If, for some $0<p<q$, there exists a non trivial joining of the flows
  $h^\tau_{pt}$ and $h^\tau_{qt}$, then $\tau$ is cohomologous to a
  constant. 
\end{theorem}

The goal of this note is to show that the above theorem follows easily
from Ratner classification of joinings of times changes of horocycles
flows \cite{Rat87} and the characterization of coboundaries given by the authors
in~\cite{FF03}. 

A slightly weaker version of Theorem \ref{thm:ratner_cohomology:1} was first proved by Kanigowski, Lema{\'n}czyk and Ulcigrai~\cite{KLU} as a consequence of a general disjointness criterion base on the so-called Ratner property. 
They proved that if distinct powers of a time-change have a non-trivial joining, then the time change function is
cohomologous to a function coming from a harmonic form.  They asked us whether the result could be derived directly from Ratner's work combined with our description of the cohomology of horocycle flows \cite{FF03}. In this note we answer their question affirmatively by proving a slightly stronger result, which holds for all smooth time-changes. It should be noted that in~\cite{KLU}   the proof that the  conditions of the disjointness criterion are satisfied by time-changes of horocycle flows is based on the asymptotic of ergodic averages for horocycle flows given by A.~Bufetov and the second author in \cite{BF14}, which is a refinement of similar asymptotic results of \cite{FF03}. These asymptotic results are in turn based on the study of
the cohomology of horocycle flows and on renormalization by the geodesic flow.

%An alternative proof of this statement has been given
%by Kanigowski, Lema{\'n}czyk and Ulcigrai~\cite{KLU}.

\smallskip
Theorem~\ref{thm:ratner_cohomology:1} implies that the flow $h^\tau_t$
satisfies the \emph{AOP property}\footnote{AOP stands for
  ``Asymptotical orthogonal powers''.} if the function $\tau$ 
is not cohomologous to a constant. The AOP property was introduced in
the paper \cite{Ab-Le-Ru0} by El Abdalaoui, Lemańczyk and de la Rue in
order to advance in the study of Sarnak's \emph{Möbius orthogonality
  conjecture}(\cite{Sa}).
% For a  measurable totally ergodic
% automorphism $S$ of a standard probability space $(Y,m)$, this property states
% that all joinings of $S^p$ and $S^q$ are weakly converge to the product joining 
% when the co-prime natural numbers $p$ and $q$ diverge.
These three 
authors prove, in the quoted article, the following consequence of the
AOP property: let $(X,T)$ be a topological uniquely ergodic dynamical
system measurably conjugated to a measurable totally ergodic
automorphism $S$ of a standard probability space $(Y,m)$; then $(X,T)$
satisfies the Möbius orthogonality conjecture: for any continuous
function $f\in C(X)$ of average zero and any $x\in X$, the  Möbius
function~$\mu$ satisfies the identity
\[
\lim_{N} \frac 1 N \sum_{i=0}^{N-1}f(T^ix)\mu(i)=0.
\]

Thanks to this work we can conclude, from the above theorem, that 
the Möbius orthogonality conjecture holds for all non-trivial smooth time-changes
of horocycle flows on compact quotients:
\begin{theorem} \label{thm:Moebius}
Let
$\tau \in W^s(\Gamma\backslash G)$ be a  strictly positive 
function of Sobolev order $s>2$, not cohomologous to a constant, and let 
$h^\tau_t$ be the corresponding time change of the horocycle flow $h_t$.  
Any topological uniquely ergodic dynamical  system measurably conjugated 
to the time-one map $h^\tau _1$ of the flow $h_t^\tau$ satisfies the Möbius 
orthogonality conjecture.
\end{theorem} 
For horocycle flows on compact quotients the Möbius orthogonality conjecture was proved J.~Bourgain, P.~Sarnak and T.~Ziegler~\cite{BSZ13}. Hence it also holds for trivial time changes with continuous transfer function, as they are topologically conjugated to the horocycle flow. To the authors best knowledge, it is an open question whether the Möbius orthogonality conjecture holds for all trivial time changes, that is, for all time-changes with measurable or square-integrable transfer function. However, it follows from the authors results in \cite{FF03} that, within the space of measurably trivial time-changes, the subspace of those with continuous transfer function
has finite codimension.

\section{Setting} 
In the following we shall have $G=\operatorname{PSL}_2(\R)$, and
$\Gamma, \widehat \Gamma< G$ co-compact lattices. We denote
$K= \operatorname{PSO}(2)<G$, the usual maximal compact subgroup of
$G$.

The following matrices form a basis Lie algebra $\sl(\R)$ of~$G$:
\[
  X= \tfrac 1 2
  \begin{psmallmatrix}
    1&0\\
    0&-1
  \end{psmallmatrix}, \quad U=
  \begin{psmallmatrix}
    0&1\\
    0&0
  \end{psmallmatrix}, \quad \bar U=
  \begin{psmallmatrix}
    0&0\\
    1&0
  \end{psmallmatrix}
\]
Then $\Theta= (U-V)/2$ is the generator of $K$.

The flows $g_\sigma$ and $h_t$ on quotients~$\Gamma \backslash G$
respectively given by right multiplication by the one-para\-meter
groups~$e^{\sigma X}$ and~$e^{tU}$ are, by definition, the (classical)
geodesic and horocycle flow on~$\Gamma \backslash G$. They preserve
the probability measure~$\mu$ on~$\Gamma \backslash G$ locally defined
by a Haar volume form.

Let $\Delta$ and $\square$ be the elements of the enveloping algebra
of $\sl(\R)$ defined by $\Delta=-X^2-U^2/2 -V^2/2$ and
$\square=-X^2-UV/2 -VU/2$. Then $\Delta$ is positive definite and
coincides on $K$-invariant function with the Laplace-Beltrami operator
for the hyperbolic metric on the Riemann surface
$\Gamma \backslash G/K$.

For any unitary representation of $G$ on a Hilbert space $H$ we denote
by $W^s(H)$ the space of Sobolev vectors of order $s\in \R^+$,
\textit{i.e.}\ the closed domain of the operator
$(1+\Delta)^{s/2}$. The space $W^s(H)$ is a Hilbert space for the norm
$\norm{f}_s :=\langle (1+\Delta)^s f, f\rangle_{H}$. If $H$ decomposes
as a direct Hilbert sum $H=\bigoplus_{\alpha \in I}H_\alpha$ of
$G$-invariant closed subspaces $H_\alpha$, then the Sobolev space
$W^s(H)$ also splits as a Hilbert direct sum
$W^s(H)= \bigoplus_{\alpha \in I}W^s(H_\alpha)$ of the mutually
orthogonal and $G$-invariant Sobolev spaces $W^s(H_\alpha)$.  We write
$W^s(\Gamma\backslash G)$ for $W^s(L^2(\Gamma\backslash G))$.  Clearly
$0\le s < t$ implies a continuous embedding of $W^t(H)$ into $W^s(H)$.

As $G$ acts on $\Gamma\backslash G$ preserving the measure $\mu$, we
have a unitary representation of $G$ on $L^2(\Gamma\backslash G)$. It
is well known that this representation decomposes as a Hilbert sum of
mutually orthogonal irreducible (or primary) sub-representations. We
recall that these are parametrized by the spectrum of the Casimir
operator~$\square$ previously defined\footnote{Functions in
  $L^2(\Gamma\backslash G)$ which are $K$-invariant are naturally
  identified with function on the Riemann surface
  $\Gamma\backslash G/K$. As in \cite{FF03}, the Casimir operator
  $\square$ is normalized so to coincide, on these functions, with the
  Laplacian-Beltrami operator of the Riemannian metric of curvature
  $-1$ on $\Gamma\backslash G/K$.}.

In fact, the spectrum of the Casimir $\square$ consists of finitely
many values in the interval $(0,1/4)$, infinitely many countable real
values in the interval $[1/4, \infty)$ and the integer values
$-n^2 +n$, with $n=1,2, \dots$. (Such subdivision of the spectrum
correspond to the classification of irreducible unitary
representations of $G$ into complementary, principal and discrete
series.)

Given a unitary representation of $G$ on a Hilbert space $H$, the
space of distribution of Sobolev order $s\ge 0$ is, by definition, the
space $W^{-s}(H)$ dual to $W^s(H)$. A distribution $D\in W^{-s}(H)$ is
invariant for the horocycle flow if $D(f \circ h_t)=D(f)$ for all $t\in \R$,
and for any
$f\in W^s(H)$. The Sobolev order of a distribution
$ D \in W^{-\infty}(H):=\bigcup_{s\ge 0}W^{-s}(H) $ is the extended
real number
\[
  s_D := \inf \{s \ge 0 \mid D \in W^{-s}(H)\}.
\]
We recall the following theorem proved in \cite{FF03}.
\begin{theorem}\label{thm:ratner_cohomology:2}
  Let $H$ be a Hilbert space on which $\operatorname{PSL}_2(\R)$ acts
  by a unitary irreducible non-trivial representation~$\rho$. The
  subspace $\mathcal I(H)$ of \(h_t\)-invariant
  distributions in $W^{-\infty}(H)$ has dimension one or two. More
  precisely:
  \begin{itemize}
  \item If $\rho$ belongs to the principal series $\mathcal I(H)$ is
    spanned by two distributions of Sobolev order $1/2$.
  \item If $\rho$ belongs to the complementary series\footnote{At most
      a finite number of such representations occur in
      $L^2(\Gamma\backslash G)$.} $\mathcal I(H)$ is spanned by two
    distributions of Sobolev order $1/2 -\delta$ and $1/2 +\delta$,
    with $\delta <1/2$.
  \item If $\rho$ belongs to the discrete series $D_{2n}$ with Casimir
    value $\square=-n^2+n$ and $n\in \{1, 2,\dots\}$, the space
    $\mathcal I(H)$ is a one-dimensional formed by distributions of
    Sobolev order $n$.
  \end{itemize}
  Furthermore, the space $\mathcal I(H)$ is invariant under the action
  of the geodesic flow $g_s$. If\/ $D\in \mathcal I(H)$, with Sobolev
  order $s_D$, then, for some $\lambda_D$ depending on the Casimir
  value $\square$, we have
  \[
    (g_u)_* D= e^{-\lambda_D u} D,\quad \text{ with } \quad\Re
    \lambda_D = s_D,
  \]
  unless $\square =1/4$. If \(\square =1/4\)
  there is a basis of $\mathcal I(H)$, for which the matrix of
  $(g_u)_* $ on $\mathcal I(H)$, is given by
  \[
    e^{-u/2}\begin{pmatrix} 1 &u\\ 0& 1
    \end{pmatrix}.
  \]
\end{theorem}

We recall that a function $f$ on $\Gamma\backslash G$ is a co-boundary
for the horocycle flow $h_t$ with a primitive
$g\in W^s(\Gamma\backslash G)$ if, for all $s\in\R$, we have
\[
  g\circ h_t - g = \int_0^t f \circ h_u \,\D u.
\]
Clearly a smooth co-boundary with a smooth primitive is in the kernel
of all $h_t$-invariant distributions. The following theorem
follows immedately from \cite[Theorem 1.3]{FF03}. 
% \begin{theorem}\cite[Theorem 1.3]{FF03}
%   There exists $s_0\in [1/2,1)$, depending on $\Gamma$ such that the
%   following holds for all $s>s_0$. If $g\in W^s(\Gamma\backslash G)$
%   is a co-boundary for the horocycle flow with a primitive
%   $f\in W^t(\Gamma\backslash G)$ then $D(g)=0$ for all $h_t$-invariant
%   distributions~$D$ of Sobolev order satisfying $s_d<s$ and $s_D\le  t+1$.
% \end{theorem}
\begin{theorem}
\label{thm:ratner_cohomology:3}
  For any  $s>1$, if $f\in W^s(\Gamma\backslash G)$
  is a co-boundary for the horocycle flow with a primitive
  $g\in L^2(\Gamma\backslash G)$, then $D(g)=0$ for all $h_t$-invariant
  distributions~$D$ of Sobolev order $s_D\le 1$.
\end{theorem}

Vanishing of $h_t$-invariant distributions characterizes the space of
co-boundaries. More precisely,   \cite[Theorems 1.1 and 1.2]{FF03} implies the following.

\begin{theorem}\label{thm:ratner_cohomology:4}
  Let $s>1$. Suppose $f\in W^s(\Gamma\backslash G)$ satisfies $D(f)=0$
  for all $h_t$-invariant distributions~$D$ of Sobolev order
  $s_D\le 1$. Then $f$ is a co-boundary with primitive  $g\in W^t(\Gamma\backslash G)$
  for any $t<\min(1,s-1)$.
\end{theorem}

The above theorem needs a bit of explaining, as it is not formulated as
such in~\cite{FF03}. By Theorem~\ref{thm:ratner_cohomology:2} above, for any
$s_1\in (1,2)$, the space $\mathcal I^{s_1}$ of $h_t$-invariant
distribution of order $\le s_1$ coincides with $\mathcal I^{1}$. By
hypothesis, there exists $s_1\in (1,2)$ such that $s_1\le s$;
consequently $g\in W^{s_1}(\Gamma\backslash G)$. Then
\cite[Theorem 1.2]{FF03} implies that $g$ is a co-boundary with a
primitive in $W^{s_1-1-\epsilon}(\Gamma\backslash G)$, for any
$\epsilon>0$.

The main tool on the proof of Theorem~\ref{thm:ratner_cohomology:1} is
the classification of the joinings of time changes of horocycle flows
proved by Marina Ratner. Ratner's theorem states
\begin{theorem}[Ratner, \cite{Rat87}]\label{thm:ratner_cohomology5} Let
  $\tau_1, \tau_2 \in C^1(\Gamma\backslash G)$ be strictly
  positive functions and suppose that they have the
  same average with respect to the Haar measure. If there exists a
  non-trivial joining of the reparametrized flows $h^{\tau_1}_t$ and
  $h^{\tau_2}_t$ then there exists a finite index subgroup
  $\widehat{\Gamma}< \Gamma$ and finite $G$-equivariant covers
  \[
    p_1\colon \widehat\Gamma\backslash G \to \Gamma\backslash G
    \quad\text{ and }\quad p_2\colon \widehat\Gamma\backslash G \to
    \Gamma\backslash G
  \]
  such that the function
  \[
    \tau_1\circ p_1 - \tau_2 \circ p_2
  \]
  is cohomologous to $0$ for the flow $h_t$ on
  $\widehat\Gamma\backslash G$. Under the assumption that 
  $\tau_1$, $\tau_2 \in W^s(\Gamma\backslash G)$ with $s>2$
 the primitive of the function $\tau_1\circ p_1 - \tau_2 \circ p_2$ belongs
  to $W^{t}(\widehat\Gamma\backslash G)$, for all $t< 1$.
\end{theorem}

We remark that, by the Sobolev embedding theorem, the theorem holds under the assumption that
$\tau_1, \tau_2 \in W^s(\Gamma\backslash G)$  for any $s>5/2$. As the original theorem  yields only a {\it measurable} primitive, the last part of the statement requires the cocycle rigidity lemma proved below.

\begin{lemma} 
A function  $f \in W^s(\Gamma\backslash G)$ with $s>2$ is a co-boundary for the horocycle flow $h_t$ with
measurable primitive if and only if it is a co-boundary with primitive in $W^t (\Gamma\backslash G)$,
for all $t<1$.
\end{lemma} 
\begin{proof} Since the converse implication is immediate, let us assume that $f$ is a co-boundary with measurable primitive and derive that in fact the  primitive  belongs to $W^t (\Gamma\backslash G)$,
for all $t<1$.  By Luzin's theorem, for any $\epsilon >0$, there exists a constant $C_\epsilon>0$ such that the 
following holds. For any $t>0$ there exists a measurable set $\mathcal  B_{\epsilon, t} \subset \Gamma\backslash G)$ volume $\text{vol} (B_{\epsilon, t}) \geq 1-\epsilon$ such that
\begin{equation}
\label{eq:meas_cob}
\vert \int_0^t  f (h_u x) du  \vert \leq C_\epsilon  \,, \quad \text{ for all } \, x \in B_{\epsilon, t}\,.
\end{equation}

If the function $f$  is not in the kernel of all horocycle invariant distributions supported on irreducible representations of the principal and complementary series,  it follows from the results of the authors~\cite[Theorem 5]{FF03} that the $L^2$ norm of ergodic integrals $f$ diverge (polynomially) and every weak limit of the random variables
$$
\frac{ \int_0^t  f (h_u x) du} { \Vert \int_0^t  f (h_u x) du \Vert_{L^2(\Gamma\backslash G)} } \,,
$$
in the sense of probability distributions, is a (compactly supported) distribution on the real line {\it not supported
at the origin}. The argument is given in detail in the proof of \cite[Corollary 5.6]{FF03} which established that
the Central Limit Theorem does not hold for horocycle flows. Refined theorems on limit distributions for horocycle
flows were proved in \cite{BF14} but under the slightly stronger (technical) assumption that $f\in W^s(\Gamma\backslash G)$ with $s>11/2$.  It follows that in this case property \eqref{eq:meas_cob} implies that $f$ is in
the kernel of all invariant distributions supported on on irreducible representations of the principal and complementary series. 

It follows then from the results of \cite{FF03} (in particular, from the formulas for invariant distributions of sections 3.1 and 3.2, and from Theorem \ref{thm:ratner_cohomology:4} above) that $f$ is cohomologous to a function given by a harmonic form on the unit tangent bundle, with a primitive in $W^t(\Gamma\backslash G)$ for all $t<1$.  Our argument is then reduced to the proof that a non-zero harmonic form cannot be a co-boundary with measurable primitive. Indeed, this statement follows from results of D.~Dolgopyat and O.~Sarig \cite{DS17} on the so-called  {\it windings} of the horocycle flow, for instance, from \cite[Theorems 3.2 and  5.1]{DS17} or  from \cite[Lemma 5.10]{DS17}.

\end{proof} 

\section{Proof of Theorem 1}

Let $\tau \in W^s(\Gamma\backslash G)$ be a  strictly positive 
function of Sobolev order $s>2$, and let $h^\tau_t$ be the corresponding
time change of $h_t$. Suppose $0<p<q$.

The flows $h^{\tau}_{pt}$ and $h^{\tau}_{qt}$ can be rewritten as the
flows $h^{\tau'}_{t}$ and $h^{\tau''}_{t}$ with $\tau''=\tau/p$ and
$\tau''=\tau/q$.

Now setting $\tau_1= \tau\circ g_{\sigma_1}$ with $\sigma_1=(\log p)/2 $ and
$\tau_2= \tau\circ g_{\sigma_2}$ with $\sigma_2=(\log q)/2 $ we have
$g_{\sigma_1} \circ h^{\tau}_{pt} = g_{\sigma_1} \circ h^{\tau'}_t =
h^{\tau_1}_t \circ g_{\sigma_1}$ and $g_{\sigma_2} \circ h^{\tau}_{qt} = g_{\sigma_2} \circ h^{\tau''}_t =
h^{\tau_2}_t \circ g_{\sigma_2}$. 

Let $\rho$ be a joining of the flows $h^{\tau}_{pt}$ and
$h^{\tau}_{qt}$.  Then the measure
$\rho_{p,q}=(g_{\sigma_1}\times g_{\sigma_2})_*\rho$ is a joining of the flows
$h^{\tau_1}_t$ and $h^{\tau_2}_t$.

By Ratner's Theorem~\ref{thm:ratner_cohomology5}, if $\rho_{p,q}$ is not the product joining, there exist a
finite index subgroup $\Gamma_0$ of $\Gamma$ and two coverings
$p_i: \Gamma_0\backslash G \to \Gamma\backslash G $ commuting with the
right action of $G$, such that the functions $ \tau_i \circ p_{i}$ are
cohomologous for the horocycle flow $h_t$ on $ \Gamma_0\backslash G$.

In conclusion, the functions $\tau \circ g_{\sigma_1} \circ p_1$ and
$\tau \circ g_{\sigma_2}\circ p_2$, are cohomologous for the
horocycle flow on $\Gamma_0\backslash G $. Equivalently,
$\tau\circ p_1$ and $\tau \circ p_2\circ g_{\sigma_2-\sigma_1}$ are
cohomologous functions for the horocycle flow on
$\Gamma_0\backslash G $. Since $p<q$ we have $\sigma:= \sigma_2-\sigma_1>0$.

Denote by $\mathcal I^s(\Gamma\backslash G) $
and $\mathcal I^s(\Gamma_0\backslash G) $ the space of $h_t$-invariant
distributions of order~$\le s$ on the respective spaces.  Our goal is to
show that
\begin{equation}
  \label{goal}
  D(\tau )=0
\end{equation}
for all $h_t$-invariant distributions $D\in \mathcal I^1(\Gamma\backslash G)$.

What we know, by theorem~\ref{thm:ratner_cohomology:3}, is that
\begin{equation}
  \label{zero}
  D(\tau\circ p_1- \tau \circ p_2\circ g_{\sigma_2-\sigma_1})=0
\end{equation}
for all $D\in \mathcal I^1(\Gamma_0\backslash G)$, as
the function $\tau\circ p_1- \tau \circ p_2\circ g_{\sigma_2-\sigma_1}$ is a
co-boundary with a primitive in $W^t(\Gamma_0\backslash G)$ for any
$t< \min (1,s-1)$.

Remark that Sobolev norms on
$\Gamma\backslash G $ are defined by the local right $G$ action on
$\Gamma\backslash G $.  The pullback operators
\[p_i^* \colon W^s (\Gamma\backslash G) \to W^s
  (\Gamma_0\backslash G), \qquad p_i^* (f) = f \circ p_i,\quad f \in
  W^s (\Gamma\backslash G),
\]
intertwine the $G$ actions on these spaces; hence they preserve the
norms in any Sobolev space, i.e.\ they are isometric embeddings of
$W^s (\Gamma\backslash G)$ onto $G$-invariant subspaces
$V_i\subset W^s (\Gamma_0\backslash G)$. Regarding $p_i^* $ as
a $G$-equivariant isomorphims
$p_i^*\colon W^s (\Gamma\backslash G) \to V_i$ we have dual
$G$-equivariant isometric isomorphims
$(p_i)_*\colon V_i' \to W^{-s} (\Gamma\backslash G) $, which is
the restriction to $ V_i'$ of the surjective map $(p_i)_*$.

Consider the the orthogonal $G$-invariant decomposition
\[
  W^s (\Gamma_0\backslash G) = V_i\oplus {V_i}^\perp.
\]
and the associated orthogonal projections $\pi^{V_i}$,
$\pi^{{V_i}^\perp}$ and inclusions $j^{V_i}$, $j^{{V_i}^\perp}$.

Define a $G$-equivariant linear isometric immersion
\[
  J_i\colon W^{-s} (\Gamma\backslash G) \to W^{-s}
  (\Gamma_0\backslash G)
\]
by
\[
  J_i(D)(f) =\begin{cases}
    D(g), &\text{if } f\in V_i, \text{ with } f=p_i^* g \quad ( g\in  W^{-s} (\Gamma\backslash G)) ;\\\
    0&\text{if } f\in V_i^\perp.
  \end{cases}
\]

Clearly $(p_i)_* \circ J_i(D)(g)= D(g)$ for all
$g\in W^{-s} (\Gamma\backslash G) $, i.e.\
$(p_i)_* \circ J_i(D)= D$.

Set
\[
  P_1=(p_2)_* \circ J_1, \qquad P_2=(p_1)_* \circ J_2.
\]

By definition both these operators are weak contractions since they
are compositions of a contraction and of an isometry.
 
As both maps $P_i$ intertwine the $G$-action, they map into themselves
the spaces $\mathcal I(\mu)$ of $h_t$-invariant distributions of a
given Casimir parameter $\mu$, and also the generalised eigenspaces of
$(g_\sigma )_*$ on $\mathcal I(\mu)$ (Jordan blocks for $(g_\sigma )_*$ appear
only when $\mu=1/4$).

Let $\mathcal I_\alpha \subset \mathcal I(\mu)$ be the eigenspace
defined by
\[
  (g_\sigma )_*D= e^{\alpha \sigma} D.
\]
Recall that we always have $\Re (\alpha)<0$. 
Suppose $D\in \mathcal I_\alpha$ is an eigenvector of $P_1$ of
eigenvalue $\lambda$ with $|\lambda|\le 1$. Then, for all $\sigma\in \R$, 
\[
  \begin{split}
    D(\tau) &=J_1 (D)(\tau\circ p_1) = J_1 (D)(\tau\circ p_2 \circ
    g_\sigma)\\& = e^{\alpha \sigma} J_1 (D)(\tau\circ p_2) = e^{\alpha \sigma} P_1
    (D) (\tau)= \lambda e^{\alpha \sigma} D (\tau).
  \end{split}
\]
This implies that $D(\tau)=0$, unless $\lambda e^{\alpha \sigma} =1$. In
particular, since $|\lambda|\le 1$ and  $\sigma \Re (\alpha)< 0$, this implies $D(\tau)=0$. Let $\mathcal E\subset 
\mathcal I_\alpha$ be a $P_1$-invariant
subspace of maximal dimension such that $ E(\tau)=0$ for all $E\in \mathcal E$.
We claim that $\mathcal E=\mathcal I_\alpha$.  Otherwise $P_1$ has an
eigenvector $D+ \mathcal E$ in the quotient space
$\mathcal I_\alpha/\mathcal E$, of eigenvalue $\lambda$ with
$|\lambda|\le 1$. Then, for some $E\in \mathcal E$,
\[
  \begin{split}
    D(\tau) & = J_1 (D)(\tau\circ p_1) = J_1 (D)(\tau\circ p_2 \circ
    g_\sigma)\\& = e^{\alpha \sigma} J_1 (D)(\tau\circ p_2) = e^{\alpha \sigma} P_1
    (D) (\tau)= e^{\alpha \sigma} ( \lambda D + E) (\tau) =  \lambda e^{\alpha \sigma} 
    D (\tau)
  \end{split}.
\]
From this follows, as above that $D(\tau)=0$, proving the claim.

If $\mu=1/4$ we have proved that, for all $D$ belonging to the
eigenspace $\mathcal I_{-1/2}\subset \mathcal I(1/4)$, we have
$D(\tau)=0$. Since $\mathcal I_{-1/2}$ is $P_1$ invariant, we can pass
to the quotient space $\mathcal I(1/4)/\mathcal I_{-1/2}$. The action
of $g_\sigma$ on this subspace is diagonalizable (with eigenvalue
$e^{-\sigma/2}$).

Let $D+ \mathcal I_{-1/2}\in \mathcal I(1/4)/\mathcal I_{-1/2}$ be an
eigenvector of $P_1$ of eigenvalue $\lambda$ with $|\lambda|\le
1$. Then, for some $E, E'\in \mathcal I_{-1/2}$
\[
  \begin{split}
    D(\tau) &=J_1 (D)(\tau\circ p_1) = J_1 (D)(\tau \circ g_\sigma\circ p_2
    )\\& = P_1 (D)(\tau\circ g_\sigma) = \lambda D(\tau\circ g_\sigma) +E
    (\tau\circ g_\sigma) \\&= [e^{-\sigma/2} \lambda D + E'](\tau) + e^{-\sigma/2}
    E(\tau)= \lambda e^{ -\sigma/4} D (\tau).
  \end{split}
\]
This proves that $D(\tau)=0 $ for all $D\in \mathcal I(1/4)$ which are
eigenvectors of $P_1$ on $\mathcal I(1/4)/\mathcal I_{-1/2}$. We can
prove, in a similar way as above, that the maximal subspace of
$\mathcal I(1/4)/\mathcal I_{-1/2}$ annihilating $\tau$ coincides with
$\mathcal I(1/4)/\mathcal I_{-1/2}$, and conclude that $D(\tau)=0 $
for all $D\in \mathcal I(1/4)$.

\bibliography{ratner_cohomology} \bibliographystyle{amsplain}

\end{document}